\documentclass{amsart}
\usepackage{amsthm,stmaryrd}
\usepackage{amssymb}
\usepackage{epsfig}
\usepackage[usenames,dvipsnames]{color}
\usepackage{verbatim}
\usepackage{hyperref}
\usepackage{mathrsfs}
\usepackage[all]{xy}
\newcommand{\brk}[1]{{\left\langle{#1}\right\rangle}}

\newcommand{\ve}{\varepsilon}

\newcommand{\Nr}{{\mathsf N}}

\newcommand{\e}{{\operatorname{e}}}
\newcommand{\slt}{{\mathfrak{sl}(2)}}
\newcommand{\Ubar}{{\wb U_q^{H}\slt}}
\newcommand{\cat}{\mathscr{C}}

\newcommand{\Id}{\operatorname{Id}}
\newcommand{\bp}[1]{{\left(#1\right)}}
\newcommand{\qn}[1]{{\left\{#1\right\}}}

\newcommand{\qd}{{\mathsf d}}
\newcommand{\Hom}{\operatorname{Hom}}
\newcommand{\Spin}{{\operatorname{\mathsf Spin}}}
\newcommand{\SO}{{\operatorname{\mathsf SO}}}
\newcommand{\C}{\ensuremath{\mathbb{C}} }
\newcommand{\Z}{\ensuremath{\mathbb{Z}} }
\newcommand{\R}{\ensuremath{\mathbb{R}} }
\newcommand{\N}{\ensuremath{\mathbb{N}} }
\newcommand{\wt}{\widetilde}
\newcommand{\wb}{\overline}
\newcommand{\ms}[1]{\mbox{\tiny$#1$}}
\newcommand{\simto}{{\stackrel\sim\to}}
\newcommand{\Cp}{{\ddot\C}}

\newcommand{\spin}{\ensuremath{\C/2\Z\text{-spin}}}
\newcommand{\spinst}{{$\spin$ structure}}
\newcommand{\epsh}[2]
         {\begin{array}{c} \hspace{-1.3mm}
        \raisebox{-4pt}{\epsfig{figure=#1,height=#2}}
        \hspace{-1.9mm}\end{array}}

\newtheorem{theo}{Theorem}[section]

\newtheorem{prop}[theo]{Proposition}
\newtheorem{cor}[theo]{Corollary}

\theoremstyle{definition}
\newtheorem{defi}[theo]{Definition}
\newtheorem{rem}[theo]{Remark}

\newtheorem{exo}[theo]{Exercise}

\theoremstyle{remark}

\renewcommand{\qedsymbol}{\fbox{\thetheo}}

\newcounter{exo} \newcounter{numexercice}
\renewcommand{\theexo}{\arabic{exo}}

\begin{document}
\title[Non semi-simple  $\slt$ quantum invariants, spin case]
{Non semi-simple  $\slt$ quantum invariants, spin case} 
\author[Blanchet]{Christian Blanchet}
\address{Univ Paris Diderot, Sorbonne Paris Cit\'e, IMJ-PRG, UMR 7586 CNRS, Sorbonne Universit\'e, UPMC Univ Paris 06, F-75013, Paris, France} 
\email{blanchet@imj-prg.fr}

\author[F. Costantino]{Francesco Costantino}
\address{Institut de Recherche Math\'ematique Avanc\'ee\\
  Rue Ren\'e Descartes 7\\
 67084 Strasbourg, France}
\email{costanti@math.unistra.fr}

\author[N. Geer]{Nathan Geer}
\address{Mathematics \& Statistics\\
  Utah State University \\
  Logan, Utah 84322, USA}
\email{nathan.geer@gmail.com}

\author[B. Patureau-Mirand]{Bertrand Patureau-Mirand}
\address{Univ. Bretagne - Sud,  UMR 6205, LMBA, F-56000 Vannes, France}
\email{bertrand.patureau@univ-ubs.fr}

\begin{abstract}
  Invariants of 3-manifolds from a non semi-simple category of modules over a
  version of quantum $\slt$ were obtained by  the last three authors
  in \cite{CGP1}.  In their construction the quantum parameter $q$ is a root
  of unity of order $2r$ where $r>1$ is odd or congruent to $2$ modulo $4$.
  In this paper we consider the remaining cases where $r$ is congruent to zero
  modulo $4$ and produce invariants of $3$-manifolds with colored links,
  equipped with generalized spin structure. For a given $3$-manifold $M$, the
  relevant generalized spin structures are (non canonically) parametrized by
  $H^1(M;\C/2\Z)$. \end{abstract}

\maketitle
\setcounter{tocdepth}{3}

\section*{Introduction}
New quantum invariants of $3$-manifolds equipped with $1$-dimensional
cohomology class over $\C/2\Z$ or equivalently $\C^*$ flat connection have
been constructed in \cite{CGP1} from a variant of quantum $\slt$.  This
family of invariants is indexed by integers $r\geq 2$, $r\not\equiv 0$ mod
$4$, which give the order of the quantum parameter.  The relevant
representation category is non semi-simple, so that the usual modular category
framework does not apply and is replaced by more general {\em relative
  $G$-modular category}. The required non degeneracy condition is not
satisfied in cases $r\equiv 0$ mod $4$. In the present paper we show that the
procedure can be adapted to the remaining cases and leads to invariants of
$3$-manifolds with colored links, equipped with some generalized spin
structure. These spin structures can be defined as certain cohomology classes
on the tangent framed bundle and can be interpreted as $C^*$ flat connections
on this framed bundle.

The non semi-simple $\slt$ invariants from \cite{CGP1} have been extended to
TQFTs in \cite{BCGP}. For $r=2$, they give a TQFT for a canonical
normalization of Reidemeister torsion; in particular they recover
classification of lens spaces. In general they give new representations of
Mapping Class Groups with opened faithfulness question.  They contain the
Kashaev invariants and give an extended formulation of the volume conjecture.
Similar TQFTs in the spin case would be interesting to construct and study.

Similarly to the cohomological case studied in \cite{CGP1}, we define in this
paper a secondary invariant for empty manifolds with integral structure (here
integral means a natural $\Spin$-structure on $M$).  These secondary
invariants might be related to the $\Spin$-refinement of the
Witten-Reshetikhin-Turaev invariants, defined in \cite{KM,Tu1,Bl1} as it is the case for the cohomological case (see \cite{CGP2}).  Also for $r=4$, the
$\Spin$-refinement of the Witten-Reshetikhin-Turaev invariant is equivalent to
the Rokhlin invariant.  It would be interesting to find a geometric
interpretation of the invariants of this paper for $r=4$.

\section{Decorated $\C/2\Z$-spin manifolds}
\subsection{$\C/2\Z$-spin manifolds}
For a quick overview on classical spin structures, see \cite{Mi}.
Quantum invariants for $3$-manifolds with spin structures have been
obtained in \cite{Bl1,KM,Tu1} and extended to TQFTs in \cite{BM}.  In
\cite{Bl2,Bl3}, following a decomposition formula given in \cite{MuH},
the first author considered generalized spin structures with
coefficients in $\Z/d$, $d$ even, and studied refined quantum
invariants of the $A$ series involving those structures. Similar
refinements for general modular categories are developed in
\cite{BBC}.  In this paper we will use generalized spin structures
whose coefficient group is $\C/2\Z$ with discrete topology.
\begin{defi}
  For $n\geq 2$, the Lie group $\Spin(n,\C/2\Z)$ is defined by
  $$\Spin(n,\C/2\Z)=\frac{(\C/2\Z)_{\text{discrete}}\times \Spin(n)}{(\overline 1,-1)}\ .$$
\end{defi}
This group is a non trivial regular cover of $\SO(n)$ with Galois
group $\C/2\Z$. If we replace $\C/2\Z$ by the discrete circle
$S^1_{\text{discrete}}$ then we would get a group  $\Spin(n,S^1)$ isomorphic to $\Spin^c(n)$, but
with stronger Lie structure, i.e. the identity
$\Spin(n,S^1)\rightarrow \Spin^c(n)$ is a smooth bijection which is
not a diffeomorphism.

 Let $n\geq 2$, and $P$ be a $\SO(n)$
principal bundle over $B$. A spin structure with coefficients in
$\C/2\Z$ on $P$ is a $\C/2\Z$ regular cover of $P$ whose restriction
to a fiber is equivalent to $\Spin(n,\C/2\Z)$ over $\SO(n)$, up to
isomorphism. 
A $\C/2\Z$ regular cover over $P$ is classified by a cohomology class
$\sigma\in H^1(P;\C/2\Z)$. 
For $n\geq 3$, the condition on the fiber says that the restriction to the fiber
\mbox{$i^*(\sigma)\in H^1(\SO(n);\C/2\Z)=\Z/2$} is non zero. For $n=2$ the condition on the fiber
says that  $i^*(\sigma)=\overline 1\in H^1(\SO(2);\C/2\Z)=
\C/2\Z$.  We get the alternative definition below.
\begin{defi}[Cohomological definition]
  Let $n\geq 2$, and $P$ be a $\SO(n)$ principal bundle over $B$. A
  \spinst\ on $P$ is a class $\sigma\in H^1(P;\C/2\Z)$
  whose restriction to the fiber has order $2$.
\end{defi}
The obstruction and parametrization problem for \spinst s is encoded
in the following exact sequence:
$$0\rightarrow H^1(B;\C/2\Z)\rightarrow H^1(P;\C/2\Z)\rightarrow 
H^1(\SO(n);\C/2\Z)\rightarrow H^2(B;\C/2\Z).$$

The obstruction is the second Stiefel-Whitney class $w_2(P)\in
H^2(B;\C/2\Z)$, and if it is non-empty, the set of \spinst s supports an
affine action of $H^1(B;\C/2\Z)$.

The definition of \spinst s applies to oriented manifolds using the
oriented framed bundle. Here the choice of the Riemannian structure is
irrelevant.
\begin{defi}
  A $\spin$ manifold is an oriented manifold together with a
  \spinst\ on its oriented framed bundle.
\end{defi}
We will denote by ${\Spin}(M,\C/2\Z)$ the set of \spinst s
on the oriented manifold $M$.
\begin{rem}
  Corresponding to the inclusion map
  \mbox{$\SO(n)=\SO(\R^n)\rightarrow \SO(1+n)=\SO(\R\oplus \R^n)$} we
  get a stabilization map $P\rightarrow P\times_{\SO(n)}
  \SO(n+1)$. The associated restriction map on the set of \spinst s is
  a bijection. This is used to define a boundary map
$$\partial : {\Spin}(M,\C/2\Z)\rightarrow {\Spin}(\partial M,\C/2\Z)\ .$$
\end{rem}
If we have a $\spin$-structure $\sigma$ which is only defined on a
submanifold $N\subset M$, then the extending problem gives a relative obstruction
$w_2(M,\sigma)\in H^2(M,N;\C/2\Z)$.
\begin{defi}
  Let $M$ be an oriented manifold, $N\subset M$, and 
  $w\in H^2(M,N;\C/2\Z)$. A $\C/2\Z$-spin structure $\sigma$ on $N$ is
  said to be complementary to $w$ if and only if $w_2(M,\sigma)=w$.
\end{defi}

If $H^0(M,N)$ is trivial, then the set ${\Spin}(N,w;\C/2\Z)$ of
$\C/2\Z$-spin structures on $N$ complementary to $w$, if non empty,
still supports a natural affine action of $H^1(M;\C/2\Z)$.

\subsection{Surgery presentation of $\C/2\Z$-spin $3$-manifold}
Let $L=(L_1,\cdots, L_m)$ be an oriented framed link in $S^3$ with
linking matrix $B=( B_{ij} )_{i,j=1,m}$ and denote by $S^3(L)$ the
$3$--manifold obtained by surgery.  Then the relative obstruction
gives a bijection 
$$\psi_L :{\Spin}(S^3\setminus 
L,\C/2\Z)\rightarrow H^2(S^3,S^3\setminus L;\C/2\Z)\approx(\C/2\Z)^m\ .$$
Recall that a $\C/2\Z$-spin structure is a cohomology class on the
oriented framed bundle and can be evaluated on a framed circle.  A key
observation is that evaluation on a trivial simple curve framed by a
spanning disc is $\overline 1\in\C/2\Z$. From this we see that
$\sigma\in {\Spin}(S^3\setminus L)$ extends to $S^3(L)$ if and
only if $\psi_L(\sigma)=c=(c_j)_{1\leq j\leq m}$ satisfies the
characteristic equation
  
\begin{equation} \label{char1}
  Bc=(B_{jj})_{1\leq j\leq m}\ \mathrm{mod}\ \ 2\ .
\end{equation}
If moreover we have a link $K=(K_1,\dots,K_\nu)$ in $S^3\setminus L$,
and $\nu$ coefficients $w_1$, \dots ,$w_\nu$, representing an element 
$w\in H^2(S^3(L),S^3(L)\setminus K;\C/2\Z)\approx H_1(K)\approx
(\C/2\Z)^\nu$, then \mbox{$\sigma\in {\Spin}(S^3\setminus
  L,\C/2\Z)$} extends to an element of
${\Spin}(S^3(L)\setminus K,w;\C/2\Z)$ if and only if
$\psi_L(\sigma)=c=(c_j)_{1\leq j\leq m}$ satisfies the equation
\begin{equation} \label{char2} B(c+c')=(B_{jj})_{1\leq j\leq m}\ \
  \mathrm{mod}\ 2\ 
\end{equation} 
where $c'_j=\sum_{\nu=1}^kw_\nu\mathrm{lk}(L_j,K_\nu)$

Using that restriction maps for $\C/2\Z$-spin structures are
equivariant with respect to affine actions of cohomologies we get the
following proposition:
\begin{prop}\label{P:spin-comb}
  With the notation above, the map $\psi_L$ induces a bijection
  between ${\Spin}(S^3(L)\setminus K,w;\C/2\Z)$ and the solutions of
  Equation \ref{char2}.
\end{prop}
We have obtained a combinatorial description of
${\Spin}(S^3(L)\setminus K,w;\C/2\Z)$ for $w\in H^2(S^3(L),S^3(L)\setminus K;\C/2\Z)$.
 
Theorem \ref{Kirby-spin} is a spin version of Kirby theorem that will
relate two surgery presentations of the same $3$-manifold with
complementary $\C/2\Z$-spin structure (see also \cite{BBC}).  The
appropriate Kirby moves will be obtained in Section \ref{S:moves} by
computing the obstructions on the elementary Kirby moves.

\begin{rem}\label{Rk:paral}
  Let $M$ be a 3-manifold equipped with  a
  ${\Spin}(M;\C/2\Z)$-structure $\sigma$.  Then $\sigma$ induces a function on
  isotopy classes of smooth framed oriented simple curve in $M$ by
  $$\sigma(\gamma)=\brk{\sigma,[\gamma]},$$
  where $[\gamma]$ is the integral first homology class of the framed curve
  $\gamma$ in the oriented framed bundle of $M$. For example,
  $\sigma(\text{unknot})=\wb1\in\C/2\Z$.  Then if $M=S^3\setminus(L\cup K)$ as
  above, we have $c_i=\sigma(m_i)+\wb1$ and $w_j=\sigma(m'_j)+\wb1$ where
  $m_i$ (resp. $m'_j$) is the standard meridian of the oriented component
  $L_i$ (resp. $K_j$).  Furthermore, $\sigma$ extends to $S^3(L)\setminus K$
  if and only if for any $\gamma$ parallel to a component of $L$, $\sigma(\gamma)=\wb1$. 
\end{rem}

\subsection{$\C$-colored links}
Let $q=\e^{\frac{i\pi}r}$ where $r\in 4\N^*=\{4,8,12,\ldots\}$ (a similar version exists
for odd $r$ but the associated topological invariants depend weakly on
the spin structure).  Recall (see \cite{CGP3}) the $\C$-algebra
$\Ubar$ given by generators $E, F, K, K^{-1}, H$ and relations:
\begin{align*}\label{E:RelDCUqsl}
  KEK^{-1}&=q^2E, & KFK^{-1}&=q^{-2}F, &
  [E,F]&=\frac{K-K^{-1}}{q-q^{-1}}, & E^r&=0,\\
  HK&=KH,  & [H,E]&=2E, & [H,F]&=-2F,& F^r&=0.
\end{align*}
The algebra $\Ubar$ is a Hopf algebra where the coproduct, counit and
antipode are defined in \cite{CGP3}.  Recall that $V$ is a
\emph{weight module} if $V$ splits as a direct sum of $H$-weight spaces
and $q^H=K$ as operators on $V$.

The category $\cat$ of weight $\Ubar$ modules is $\C/2\Z$ graded (by
the weights modulo $2\Z$) that is $\cat=\bigoplus_{\wb
  \alpha\in\C/2\Z}\cat_{\wb\alpha}$ and
$\otimes:\cat_{\wb\alpha}\times\cat_{\wb\beta}\to\cat_{\wb\alpha+\wb\beta}.$
The simple projective modules of $\cat$ are indexed by the set
\begin{equation}
  \Cp=(\C\setminus\Z)\cup r\Z.  
\end{equation}
More precisely, for $\alpha\in\Cp$, the $r$-dimensional module
$V_\alpha\in\cat_{\wb{\alpha+1}}$  is the
irreducible module with highest weight $\alpha+r-1$ (be aware of the shift between the
index $\alpha$ of a module and its degree $\wb{\alpha+1}$).  The category
$\cat$ is a ribbon category and we let $F$ be the usual
Reshetikhin-Turaev functor from $\cat$-colored ribbon graphs to $\cat$
(which is given by Penrose graphical calculus).  Here the
twist acts on $V_\alpha$ by the scalar
$$\theta_\alpha=q^{\frac{\alpha^2-(r-1)^2}2}.$$

The group of invertible modules of $\cat_{\wb0}$ is generated by $\ve$
which is the one dimensional vector space $\C$ on which $E,F,K+1$ and $H-r$ act
by $0$. For $k\in\Z$, we write $\ve^k$ for the module $\ve^{\otimes
  k}$ on which $E,F,K-(-1)^k$ and $H-kr$ act by $0$.  Remark that in
$\cat$, one has $V_{\alpha+kr}=V_\alpha\otimes\ve^k$.  Furthermore, if
$V\in \cat_{\wb{\alpha}}$ then we have
\begin{equation}
  \label{eq:sigmabraid}
  F\left( \put(5,17){$\ve$}\put(18,17){$V$}\epsh{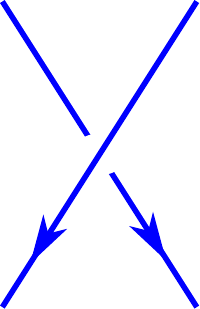}{9ex}\right)=
  q^{r \wb{\alpha}}F\left(\put(5,17){$\ve$}\put(18,17){$V$}\epsh{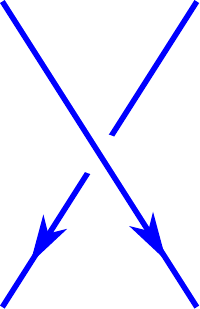}{9ex}
  \right),\ \ F\left(
    \put(5,17){$\ve$}\put(18,17){$\ve$}\epsh{fig26}{9ex}\right)
  =F\left(\put(6,17){$\ve$}\put(18,17){$\ve$}\epsh{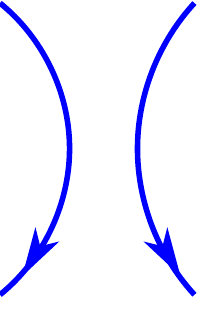}{9ex}
  \right),\ \ F\left( \put(8,10){$\ve$}\epsh{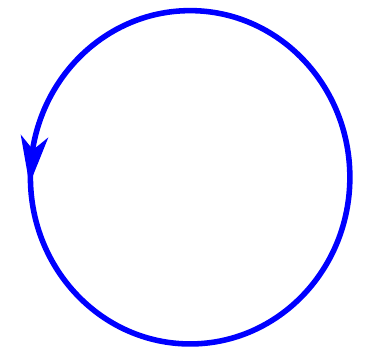}{9ex} \right)=-1.
\end{equation}

The link invariant underlying our construction is the re-normalized link
invariant (\cite{GPT}) that we recall briefly.  The modified dimension is the
function defined on $\{V_\alpha\}_{\alpha\in\Cp}$ by
$$\qd(\alpha)=-\frac{r\qn\alpha}{\qn {r\alpha}},$$
where $\qn\alpha= 2i\sin\frac{\pi \alpha}{r}$.  
Let $L$ be a $\cat$-colored oriented framed link in $S^3$ with at least one
component colored by an element of $\{V_\alpha:\alpha\in\Cp\}$.  Opening such
a component of $L$ gives a 1-1-tangle $T$ whose open strand is colored by some
$\alpha\in\Cp$ (here and after we identify $\Cp$ with the set of coloring
modules $\{V_\alpha\}$).  The Reshetikhin-Turaev functor associates an
endomorphism of $V_\alpha$ to this tangle.  As $V_\alpha$ is simple, this
endomorphism is a scalar $\brk T\in\C$.  The modified invariant is
$F'(L)=\qd(\alpha)\brk T$.
\begin{theo}[\cite{GPT}, see also \cite{ADO}]
  The assignment $L\mapsto F'(L)=\qd(\alpha)\brk T$ described above is an isotopy
  invariant of the colored framed oriented link $L$.
\end{theo}

\begin{defi}[Kirby color]
  For $\alpha\in\C\setminus\Z$, let $\wb\alpha\in\C/2\Z$ be its class
  modulo $2$.  We say that the formal linear combination of colors 
  $$\Omega_\alpha=\sum_{k=1}^{r/2}\qd(\alpha+2k-1)[\alpha+2k-1]$$ 
  is a Kirby color of degree $\wb \alpha$. 
\end{defi}

We can color a link by a formal linear combination of colorings and
expand it multilinearly.  In \cite{CGP1}, an invariant of 3-manifolds
equipped with cohomology class where constructed from the data of a
relative $G$-modular category. The category $\cat$ above failed
to be a relative $G$-modular category  because a constant
$\Delta_+$, used to compute the invariants, vanishes.  In the spin
case, the constant $\Delta_+$ can be replaced by the following $\Delta^\Spin_+$ whose
computation is similar to the one of $\Delta_+$ in \cite[Section
1.2]{CGP1}:
\begin{align}\label{E:Delta+}
  \Delta^\Spin_+&= \frac1{\qd(\alpha)}F'\bp{\epsh{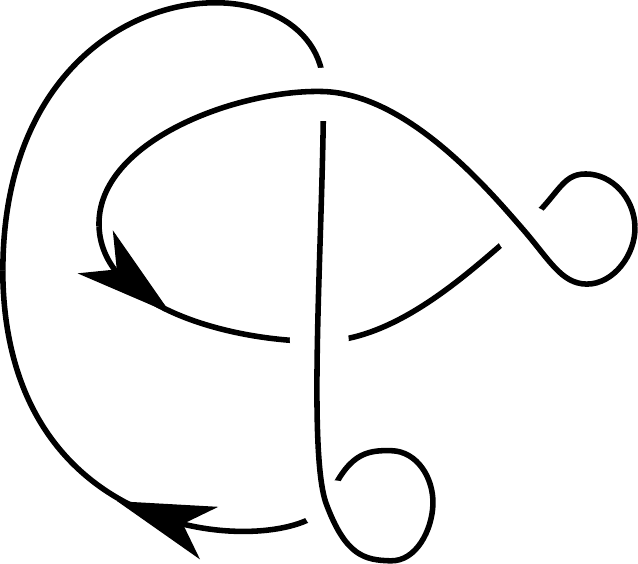}{6ex}}
  \put(-27,3){\ms{\alpha}}\put(-28,-17){\ms{\Omega_\alpha}}
  =\frac1{\qd(\alpha)}\sum_{k=1}^{r/2}\qd(\alpha+2k-1)\frac{-rq^{(\alpha+2k-1)\alpha}}{\theta_{\alpha}\theta_{\alpha+2k-1}}\notag\\
  &=\frac{r}{\qn\alpha\theta_0^2}\sum_{k=1}^{r/2}q^{\alpha+\frac12}q^{-2(k-1)^2}-q^{-\alpha+\frac12}q^{-2k^2}\notag\\
  &=rq^{\frac12+(r-1)^2}\frac{1-i}2\sqrt r=\frac{1-i}2(rq)^{\frac32}
\end{align}
where we used the value of the Gauss sum $$\sum_{k=-r/2+1}^{r/2}q^{-2k^2}
=(1-i)\sqrt r=2\sum_{k=1}^{r/2}q^{-2k^2}=2\sum_{k=1}^{r/2}q^{-2(k-1)^2}.$$

\subsection{Compatible triple}
\begin{defi}[Compatible triple]
  Let $K=K_1\cup\cdots\cup K_\nu$ be a
  $\cat$-colored 
  framed oriented link with $\nu$ components in an oriented 3-manifold
  $M$.  Let $\sigma$ be a \spinst\ on $M\setminus K$ and $c_\sigma\in
  (\C/2\Z)^\nu$ be the associated $\C/2\Z$-coloring of the components
  of $K$ ($\sigma$ is complementary to the cohomology class given by
  $c_\sigma$).  We say that $(M,K,\sigma)$ is a {\em compatible
    triple} if each component $K_j$ of $K$ is colored by a module of
  degree $c_\sigma(K_j)$.
\end{defi}

\begin{prop}\label{P:epsilon=sigma}
  Let $(S^3,L\cup K,\sigma)$ be a compatible triple where at least one
  component of $L$ is colored by an element of $\Cp$ and the component
  $K$ is an oriented ribbon knot colored by $\ve$.  Then $\sigma$
  extends uniquely to a \spinst\ on $S^3\setminus L$ and
  $(S^3,L,\sigma)$ is a compatible triple.  Furthermore,
  $$F'(L\cup K)=q^{r\sigma(K)}F'(L).$$
\end{prop}
\begin{proof}
  The obstruction to extending $\sigma$ to $S^3\setminus L$ is given
  by $c_\sigma(K)$.  But the compatibility condition for $K$ implies
  that $c_\sigma(K)=\wb0$ because $\ve\in\cat_{\wb0}$.  Then one can
  compute $\sigma(K)$ using some skein relations in $S^3\setminus L$.
  But these skein relations are precisely given by
   $$ \sigma(K)\left(\put(5,17){\ms{K}}\put(15,17){\ms{L_i}}
     \epsh{fig26}{9ex}\right)=
   \sigma(K)\left(\put(5,17){\ms{K}}\put(15,17){\ms{L_i}}\epsh{fig27}{9ex}
   \right)+c_\sigma(L_i),$$
    $$\sigma(K)\left(
      \put(5,17){\ms{K}}\put(18,17){\ms{K}}\epsh{fig26}{9ex}\right)
    =\sigma(K)\left(\put(6,17){\ms{K}}\put(18,17){\ms{K}}\epsh{fig32}{9ex}
    \right),\ \ \sigma(K)\left( \put(8,10){\ms{K}}\epsh{fig12}{9ex}
    \right)=\wb1.$$ Comparing with the skein relations
    \eqref{eq:sigmabraid}, we get the announced equality.
\end{proof}

\begin{cor}\label{C:colormodr}
  Let $L\cup K_\alpha$ be a link in $S^3$ with a component colored by
  $\alpha\in\Cp$ and let $L\cup K_{\alpha+nr}$ be the same link except
  that the color $\alpha$ is changed to $\alpha+nr$.  Then
  $$F'(L\cup K_{\alpha+nr})=q^{nr\sigma(K_\parallel)}F'(L\cup K_\alpha).$$
  where $K_\parallel$ is a framed curve parallel to $K_\alpha$.
\end{cor}
\begin{proof}
  The proof is by induction on $n$ using 
  $V_{\alpha+r}=V_\alpha\otimes\ve$ and the fact that the RT functor $F$ does
  not distinguish between a strand colored by $V\otimes V'$ and two
  parallel strands colored by $V$ and $V'$.  Then one can use
  Proposition \ref{P:epsilon=sigma}.
\end{proof}
\subsection{Link presentation of compatible triple}
\begin{defi}
  A link presentation is a triple $(L,K,\sigma)$ where 
  \begin{enumerate}
  \item $L$ is an oriented framed link in $S^3$,
  \item $K$ is a $\cat$-colored oriented framed link in $S^3\setminus L$,
  \item $\sigma$ is a \spinst\ on $S^3\setminus (L\cup K)$,
  \item the triple $(S^3\setminus L,K,\sigma)$ is compatible,
  \item $\sigma$ extends to a \spinst\ on $S^3(L)\setminus K$.
  \end{enumerate}
\end{defi}
Link presentations are regarded up to ambient isotopy.  We now give 
several important remarks on this definition:
\begin{rem}\ \label{R:linkp}
  \begin{enumerate}
  \item By the process of surgery on $L$, a link presentation gives
    rise to a compatible triple $(S^3(L),K,\sigma)$.  Reciprocally, if
    $\sigma$ is the restriction of any compatible \spinst\ on
    $S^3(L)\setminus K$ then $\sigma$ clearly satisfies (4) and (5).
  \item The \spinst\ $\sigma$ on $S^3\setminus (L\cup K)$ induces a
    $\C/2\Z$-coloring $c_\sigma$ of the components of $L\cup K$ which by (4) is
    compatible with the $\cat$-coloring of $K$.  Furthermore, Proposition
    \ref{P:spin-comb} implies that $c_\sigma$ determines $\sigma$.
  \item According to Remark \ref{Rk:paral}, the last item is equivalent to the
    fact that if $\gamma$ is a parallel to any component of $L$, then
    $\sigma(\gamma)=\wb1$.  It is also equivalent by Proposition
    \ref{P:spin-comb}, to the fact that $c$ satisfies Equation \ref{char2}
    where $c_i=c_\sigma(L_i)$ and $w_j=c_\sigma(K_j)$.
  \end{enumerate}
\end{rem}

\begin{defi}
  A link presentation $(L,K,\sigma)$ is {\em computable} if
  \begin{itemize}
  \item either $L=\emptyset$ and $K$ has a component colored by some
    $\alpha\in\Cp$,
  \item or for every component $L_i$ of $L$, $c_\sigma(L_i)\notin\Z/2\Z$.
  \end{itemize}
\end{defi}
\begin{prop}\label{P:kirbymodr}
  Let $(L,K,\sigma)$ be a computable link presentation.  We color every
  component of $L$ by a Kirby color of degree $c_\sigma(L_i)$ and compute its
  image by $F'$.  Then the complex number $F'(L,K,\sigma)$ obtained by this
  process is independent of the choice of the Kirby colors.
\end{prop}
\begin{proof}
  If a component $L_i$ of $L$ is colored by $\qd(\alpha)V_\alpha$ then
  changing it to $\qd(\alpha+nr)V_{\alpha+nr}$ does not affect $F'$.
  Indeed, by Remark \ref{R:linkp}(3), and by Corollary
  \ref{C:colormodr} the invariant changes by the sign $q^{rn}=(-1)^n$.
  But $\qd(\alpha+nr)=(-1)^n\qd(\alpha)$ so the two possible signs
  compensate.  Hence changing the Kirby color used to color $L_i$ from
  $\Omega_\alpha$ to
  $$\Omega_{\alpha+2}=\Omega_\alpha-\qd(\alpha+1)[\alpha+1]
  +\qd(\alpha+r+1)[\alpha+r+1]$$
  does not change the invariant.

  Remark for later that if, instead of $\Omega_\alpha$, one colors a
  component with the formal combination
  $\wt\Omega_\alpha=\frac12\sum_{k=1}^{r}\qd(\alpha+2k-r-1)[\alpha+2k-r-1]$,
  it does not change the invariant.
\end{proof}
\section{Moves on link presentations}\label{S:moves}
Here we present five moves on link presentations, then in Section \ref{S:TheInv} we will use these moves to show that the invariant is well defined.  In the following
pictures involving link presentations, the surgery link $L$ is colored in
red and the $\cat$-colored link $K$ is colored in blue.  A black
strand can be any component of $K\cup L$.  The first three moves are
the spin version of the usual Kirby moves.  As remarked by Gompf and
Stipsicz (see \cite[Chapter 5]{GS}), every Kirby move describes a
canonical isotopy class of diffeomorphisms between the  surgered 3-manifolds.
  The purpose of the fourth and fifth moves is to create
a generically colored knot in the 3-manifold.
\subsection{Orientation change}
$$\epsh{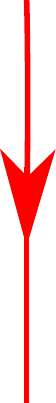}{8ex}\put(0,-3){\ms{\color{red}\wb\alpha}}\quad
\longleftrightarrow\quad\epsh{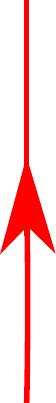}{8ex}\put(0,-3){\ms{\color{red}-\wb\alpha}}$$
\begin{defi}[Orientation move]
  The link presentations $(L,K,\sigma)$ and $(L',K,\sigma)$ are
  related by an orientation move if $L=L'$ as an unoriented framed
  link but the orientation of the components of $L'$ might differ from
  those of $L$.  Remark that the \spinst\ does not change but
  $c_\sigma(L_i)=-c_\sigma(L_i')$ if the orientation of $L_i$ has
  changed.
\end{defi}
\begin{prop}
  If two link presentations $(L,K,\sigma)$ and $(L',K,\sigma)$ are
  related by an orientation move and if they are computable then 
  $$F'(L,K,\sigma)=F'(L',K,\sigma)$$
\end{prop}
\begin{proof}
  The following changes does not affect the invariant $F'$ of
  $\cat$-colored link:
  \begin{itemize}
  \item change a $\cat$-color $V$ of $L$ with an isomorphic module, 
  \item change the orientation of a component of $L$ and simultaneously
    change its $\cat$-color with its dual.
  \end{itemize}
  Hence the proposition follows from the fact that for $\alpha\in\Cp$,
  $V_\alpha^*\simeq V_{-\alpha}$ and $\Omega_{\alpha}^*\simeq
  \Omega_{-\alpha}$.
\end{proof}
\begin{rem}
  The manifold obtained by surgery does not depends of the orientation
  of the surgery link.
\end{rem}
\subsection{Stabilization}
The usual Kirby I move is a stabilization that consists in adding to $L$ an
unknot with framing $\pm1$.  This move always leads to a non computable
presentation.  For this reason, we are led as in \cite{CGP1} to introduce a
modified stabilization.
$$\epsh{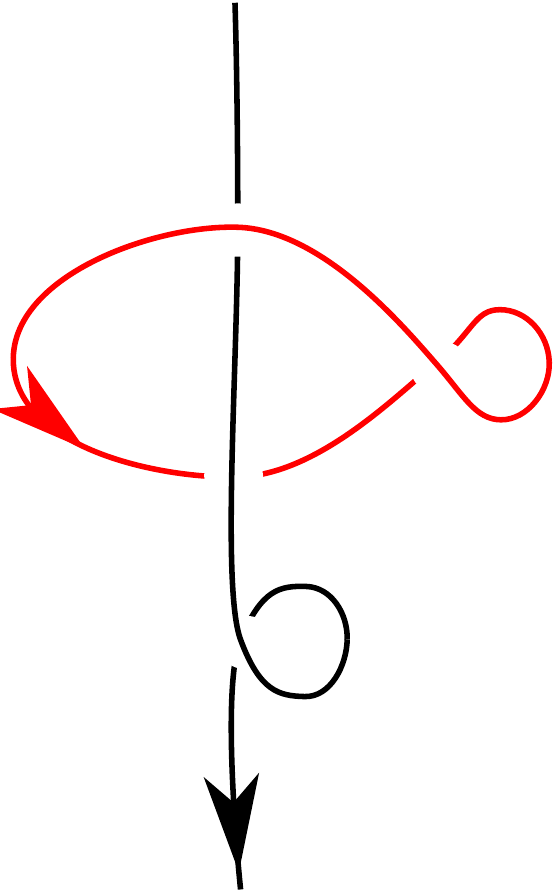}{12ex}\put(-45,-3){\ms{\color{red}\wb{\alpha+1}}}
\put(-11,-19){\ms{\wb\alpha}}
\qquad\stackrel{\text{KI}^-}\longleftrightarrow\qquad
\epsh{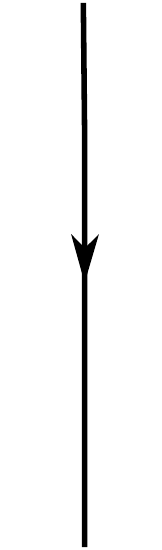}{12ex}\put(-3,0){\ms{\wb\alpha}}
\qquad\stackrel{\text{KI}^+}\longleftrightarrow\qquad
\epsh{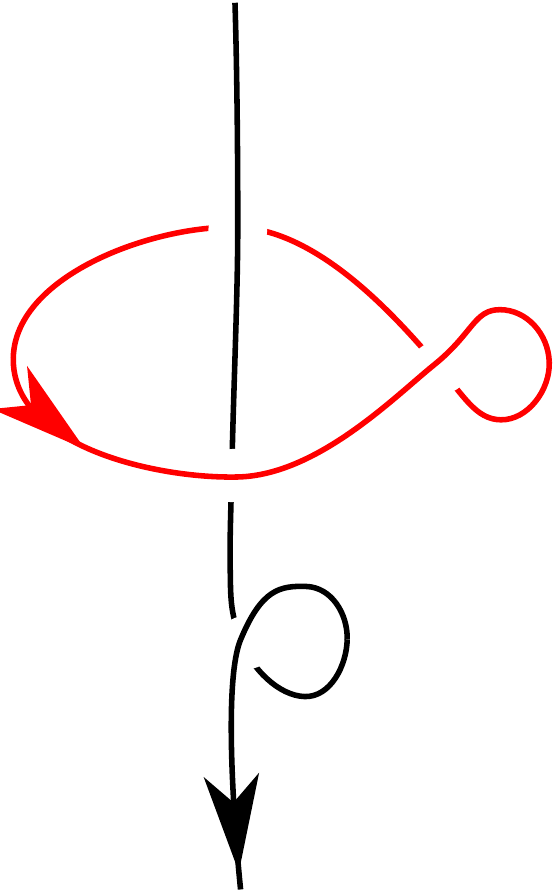}{12ex}\put(-45,-3){\ms{\color{red}\wb{\alpha+1}}}
\put(-11,-19){\ms{\wb\alpha}}$$
\begin{defi}[KI$^\pm$-move]
  The link presentation $(L'\cup o,K',\sigma')$ is obtained from the
  link presentation $(L,K,\sigma)$ by a positive KI$^\pm$-move if $o$
  is an unknot with framing $\pm1$ obtained by adding a full positive or
  negative twist to a meridian of a component $J$ of $L\cup K$.  We
  orient this meridian so that its linking number with $J$ is $\mp1$.
  The only changed component of $L\cup K$ is $J$ whose corresponding
  component $J'$ of $L'\cup K'$ is obtained by adding a full twist to
  $J$.  The \spinst\ $\sigma'$ is determined by ${c_{\sigma'}}_{|L\cup
    K}=c_\sigma$ and $c_{\sigma'}(o)=c_\sigma(J)+\wb1$.

  We say that $(L,K,\sigma)$ is obtained from $(L\cup o,K,\sigma')$ by a
  negative KI$^\pm$-move.
\end{defi}

\begin{prop}
  A KI$^\pm$-move between $(L,K,\sigma)$ and $(L'\cup o,K',\sigma')$
  induces a canonical (up to isotopy) diffeomorphism $(S^3(L)
  ,K,\sigma)\simeq(S^3(L'\cup o) ,K',\sigma')$.
\end{prop}
\begin{proof}
  The KI-move does not change $L\cup K$ outside a tubular neighborhood
  $T$ of $J$.  There is a canonical diffeomorphism $f:(S^3(L),K)\simto
  (S^3(L'),K')$.  Hence there exists a \spinst\ $\sigma'$, which is
  the image of $\sigma$ by $f$, such that $(S^3(L'),K',\sigma')$ is a
  compatible triple.  Furthermore, the diffeomorphism $f$ is the
  identity outside $T$. Thus $(L',K',\sigma')$ is a link presentation,
  furthermore $f$ send meridian of the components of $L\cup K$ to the
  corresponding meridian of the components of $L'\cup K'$ thus
  ${c_{\sigma'}}_{|L\cup K}=c_\sigma$.  Finally, the new component has
  framing $\pm1$ and is only linked $\mp1$ times with $J'$ so Equation
  \eqref{char2} implies that $c_{\sigma'}(o)=c_\sigma(J)+\wb1$.
\end{proof}
\begin{prop}
  If $(L\cup o,K,\sigma')$ is obtained from $(L,K,\sigma)$ by a positive
  KI$^\pm$-move and if they are both computable, then
  $$F'(L\cup o,K,\sigma')=\Delta^\Spin_{\pm}F'(L,K,\sigma).$$
  where $\Delta^\Spin_+=(1-i)r^{\frac32}q^{\frac32}$ and
  $\Delta^\Spin_-=(1+i)r^{\frac32}q^{-\frac32}$ is the conjugate
  complex number.
\end{prop}
\begin{proof} We give the proof for a KI$^+$-move.  The other case is
  similar.  Let $\alpha\in\C\setminus \Z$, since 
  $V_\alpha$ is simple $F\left(\,\epsh{fig1}{8ex}
    \put(-24,-3){\ms{\Omega_\alpha}}\put(-7,-12){\ms{\alpha}}\right)$
  is a scalar endomorphism of $V_\alpha$.  From Equation \eqref{E:Delta+} this scalar is equal to $\Delta^\Spin_+$.  Moreover,  this scalar times
  $\qd(\alpha)$ is equal  $F'$ of the braid closure of this 1-1
  tangle.    Now as
  before, $\Delta^\Spin_+$ does not change if we replace ${\Omega_\alpha}$
  by any Kirby color of degree $\wb\alpha$.  Furthermore, $\Delta^\Spin_+$ 
  is independent of $\alpha$ so this scalar is the same for any simple
  module of $\cat_{\wb{\alpha+1}}$.  As $\cat_{\wb{\alpha+1}}$ is semi-simple, if we replace $V_\alpha$ with any module
  $V\in\cat_{\wb{\alpha+1}}$, then this tangle evaluates to
  $\Delta^\Spin_+\Id_V$.  It follows that whatever the value of 
  $c_\sigma(J)\in\C/2\Z\setminus \Z/2\Z$ is we have $F'(L\cup
  o,K,\sigma')=\Delta^\Spin_{+}F'(L,K,\sigma).$
\end{proof}
\subsection{Handle slide}
$$\epsh{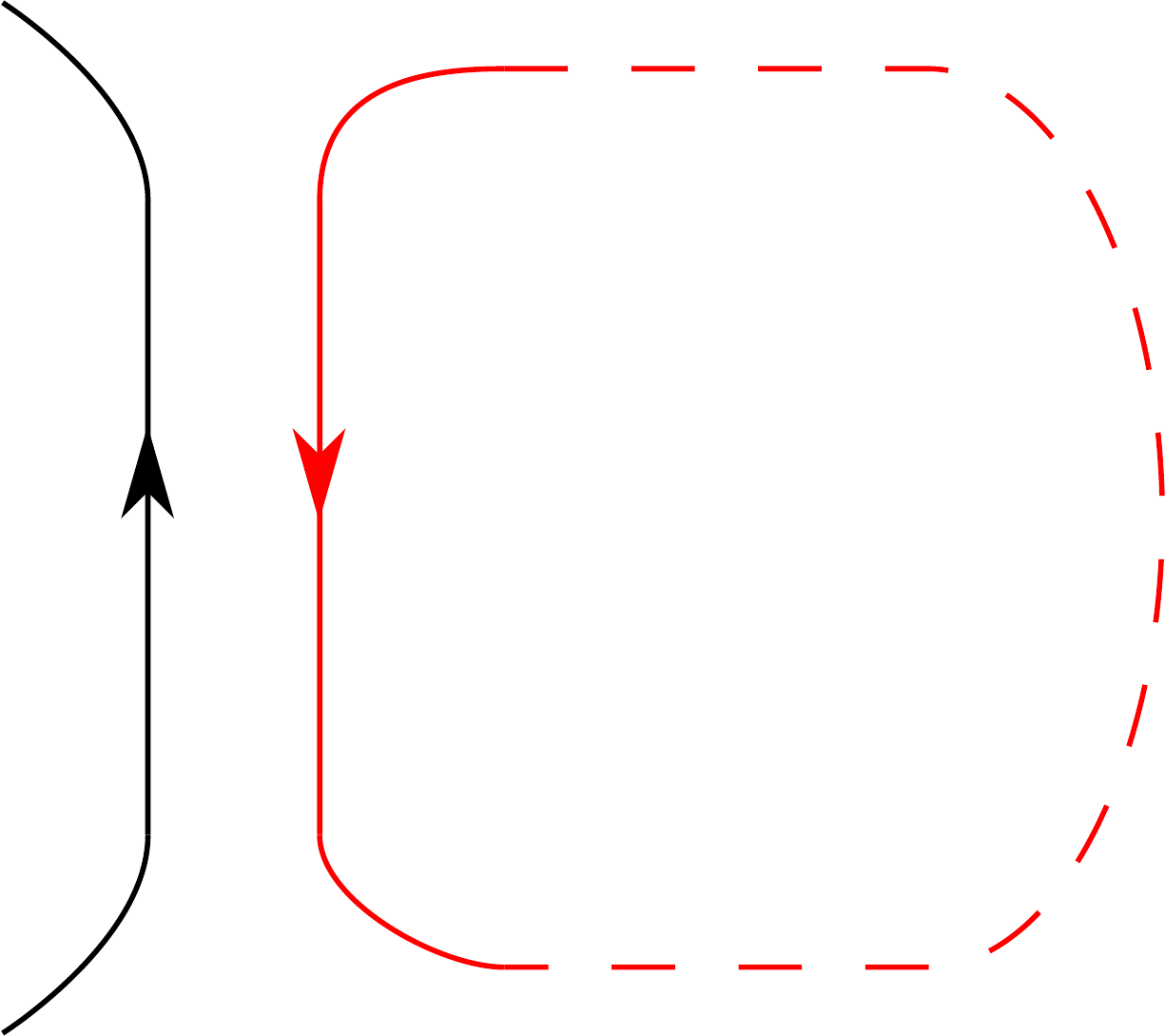}{12ex}\put(-40,1){\ms{\color{red}\wb{\alpha}}}
\put(-58,1){\ms{\wb\beta}}
\qquad\stackrel{\text{KII}}\longleftrightarrow\qquad
\epsh{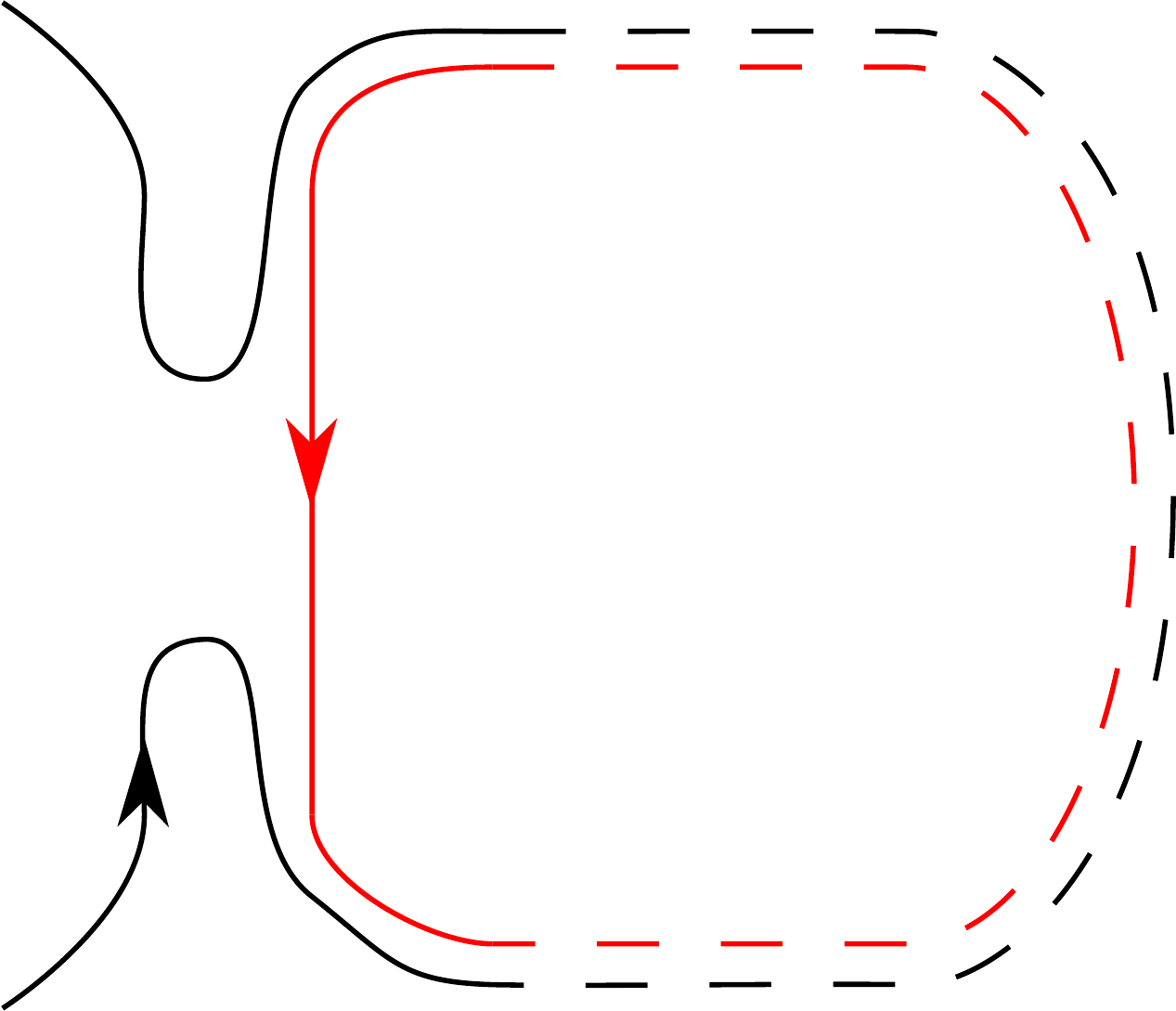}{12ex}\put(-40,1){\ms{\color{red}\wb{\alpha}-\wb\beta}}\put(-59,-17){\ms{\wb\beta}}
$$
\begin{defi}[KII-move]
  The link presentations $(L,K,\sigma)$ and $(L',K',\sigma')$ are
  related by a KII-move if 
  $L_i$ is a component of $L$ which is different from a component $J$ of $L\cup K$ 
 and $J'$ is a component of $L'\cup K'$ such
  that $L'\setminus J'=L\setminus J$ and $J'$ is a connected sum of
  $J$ with a parallel copy of $L_i$.  If $J$ is a component of $K$,
  then $J'$ and $J$ have the same $\cat$-color.  The \spinst\
  $\sigma'$ is determined by ${c_{\sigma'}}_{|(L\cup K')\setminus
    L_i'}={c_{\sigma}}_{|(L\cup K)\setminus L_i}$ and
  $c_{\sigma'}(L_i')=c_{\sigma}(L_i)-c_\sigma(J)$.
\end{defi}
\begin{prop}
   A KII-move between $(L,K,\sigma)$ and $(L',K',\sigma')$
  induces a canonical (up to isotopy) diffeomorphism $(S^3(L)
  ,K,\sigma)\simeq(S^3(L') ,K',\sigma')$.
\end{prop}
\begin{proof}
  The KII-move does not change $L\cup K$ outside a genus 2 handlebody 
  $H$ embedded in $S^3$ formed by a tubular neighborhood of $J$, of
  $L_i$ and of the path from $J$ to the parallel copy of $L_i$ used to
  make the connected sum.  There is a canonical diffeomorphism
  $f:(S^3(L),K)\simto (S^3(L'),K')$.  Hence there exists a \spinst\
  $\sigma'$, which is the image of $\sigma$ by $f$, such that
  $(S^3(L'),K',\sigma')$ is a compatible triple.  The
  diffeomorphism $f$ is the identity outside $H$, so 
  $(L',K',\sigma')$ is a presentation.  Furthermore, $f$ sends meridians
  to the components of $L\cup K$ to the corresponding meridians of the
  components of $L'\cup K'$ except for the meridian of $L_i$ which is
  sent to the connected sum of the meridian of $L'_i$ with the
  meridian of $J'$.  Hence
  $\sigma(m_{L_i})=\sigma'(m_{L'i})+\sigma'(m_{J'})+\wb1$ where the
  $\wb 1$ comes from the connected sum.  This leads to the announced
  formula for $c_{\sigma'}$.
\end{proof}
  
\begin{prop}
  If two link presentations $(L,K,\sigma)$ and $(L',K',\sigma')$ are
  related by a KII-move and if they are both computable then 
  $$F'(L,K,\sigma)=F'(L',K',\sigma')$$
\end{prop}
\begin{proof} 
  First the remark in the proof of Proposition \ref{P:kirbymodr}
  implies that instead of using a Kirby color of degree
  $\Omega_{\alpha}$ we can compute the image by $F'$ of the link
  presentation using the 
  Kirby color
  $\frac12\wt\Omega_\alpha=
  \frac12\sum_{k=1}^{r}\qd(\alpha+2k-r-1)[\alpha+2k-r-1]$.
  Then the proof is the same that in \cite[Lemma 5.9]{CGP1} where the
  module $\ve$ used in \cite{CGP1} correspond to $\ve^2$ in this paper.  
\end{proof}
\subsection{Hopf stabilization}
$$\epsh{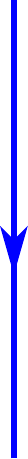}{12ex}\put(-12,-2){\ms{{\alpha_i}}}
\qquad\stackrel{\text{Hopf move}}\longleftrightarrow\qquad
\lambda.\epsh{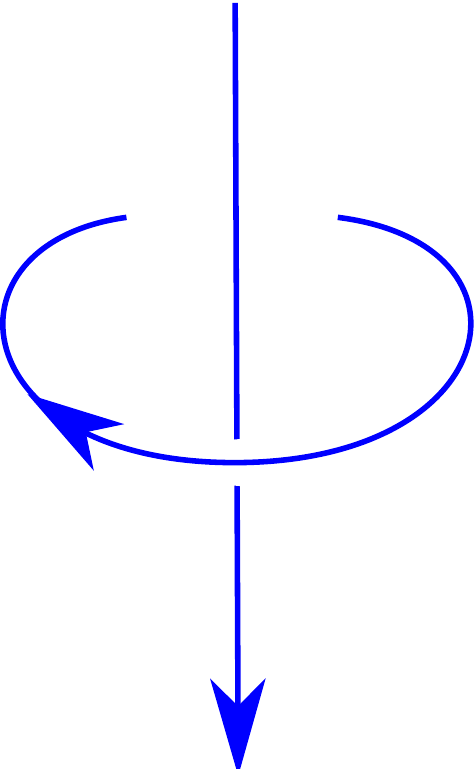}{12ex}\put(-13,-17){\ms{{\alpha_i}}}\put(-23,0){\ms{\beta}}$$
\begin{defi}[Hopf move]
  The link presentation $(L,K\cup o,\sigma')$ is obtained from the
  link presentation $(L,K,\sigma)$ by a positive Hopf move if
  $o$ is a zero framed meridian of a component $K_i$ of $K$.  The
  color $\alpha_i$ of $K_i$ has to be in $\Cp$ and the color of the
  newly added component is the coefficient $\lambda=\dfrac{\qd(\alpha_i)}{-rq^{\beta\alpha_i}}$ times a color $\beta\in\Cp$.  
\end{defi}
\begin{prop}
  If two link presentations $(L,K,\sigma)$ and $(L,K\cup o,\sigma')$
  are related by a Hopf stabilization then
  $$F'(L,K,\sigma)=F'(L,K\cup o,\sigma').$$
\end{prop}
\begin{proof}
  The simplicity of $V_\alpha$ for $\alpha\in\Cp$ implies that
  $F\left(\,\epsh{fig4}{6ex}
    \put(-18,-3){\ms{\beta}}\put(-7,-12){\ms{\alpha}}\right)$ is a
  scalar endomorphism of $V_\alpha$.  This scalar is computed (for
  example in \cite{GPT}) and is given by
  $\dfrac{-rq^{\beta\alpha_i}}{\qd(\alpha_i)}$.  It follows that
  $F'(L,K\cup o,\sigma')$ is just
  $\dfrac{-rq^{\beta\alpha_i}}{\qd(\alpha_i)}\lambda$ times
  $F'(L,K,\sigma')$.
\end{proof}
\begin{rem}
  There is a similar notion of Hopf stabilization for compatible
  triple $(M,K,\sigma)$.  The resulting compatible triple
  $(M,K\cup o,\sigma')$ can be seen as a banded connected sum of
  $(M,K,\sigma)$ with $(S^3,H,\sigma_H)$ where $H$ is proportional to
  a Hopf link in $S^3$.  Furthermore, a Hopf stabilization of a
  surgery presentation of a compatible triple is clearly a surgery
  presentation of a Hopf stabilization of the triple.
\end{rem}

\subsection{Birth move}
In the following positive move two new components appear in the $K$ part of a link presentation. 
$$\epsh{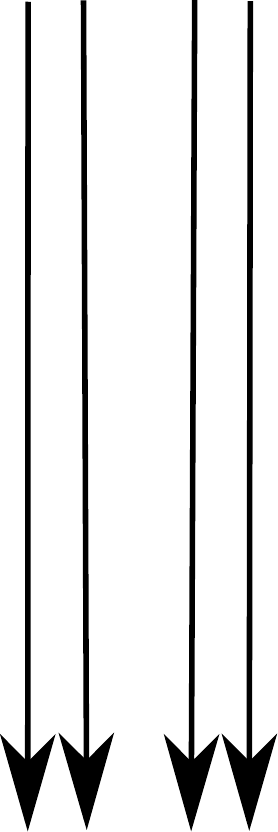}{12ex}
\qquad\stackrel{\text{Birth move}}\longleftrightarrow\qquad
\frac1{|\Delta^\Spin_+|^2}\qquad
{\put(-9,12){\ms{{\Omega_\alpha}}}\put(-8,-12){\ms{{\Omega_\alpha}}}\epsh{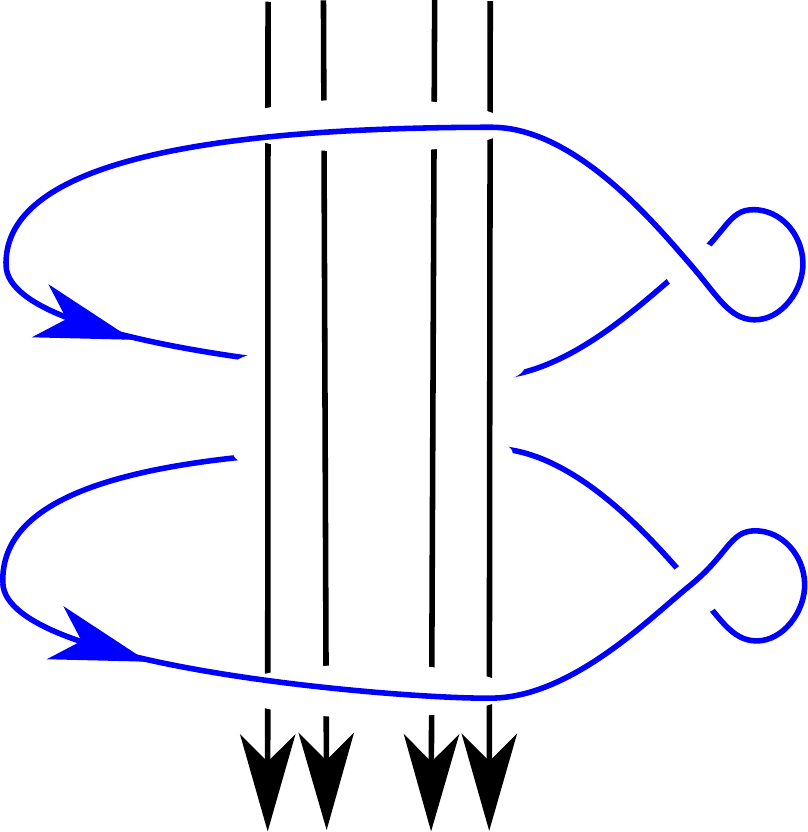}{12ex}}$$
\begin{defi}[Birth move]
  The link presentation $(L,K\cup K_+\cup K_-,\sigma')$ is obtained
  from the computable link presentation $(L,K,\sigma)$ by a \emph{positive
  birth move} if it is given by the following process: Let $D$ be an
  oriented disc in $S^3$ and $\partial D$ its oriented framed
  boundary.  We assume that $\partial D$ is in general position
  relatively to $L\cup K$ and that $\wb \alpha=\sigma(\partial
  D)\notin\Z/2Z$.  Then $K_+$ and $K_-$ are two parallel copies of
  $\partial D$ with framing $+1$ and $-1$ respectively, with the
  orientation of $K_+$ reversed and they are both colored by
  $\frac1{|\Delta^\Spin_+|}$ times a Kirby color of degree $\wb
  \alpha$.  The \spinst\ $\sigma'$ is determined by
  $(c_{\sigma'})_{|L\cup K}=c_\sigma$ and $c_{\sigma'}(K_\pm)=\wb
  \alpha$.
\end{defi}

\begin{prop}
  If two computable link presentations $(L,K,\sigma)$ and $(L,K\cup
  K_+\cup K_-,\sigma')$ are related by a birth move then
  $$F'(L,K,\sigma)=F'(L,K\cup K_+\cup K_-,\sigma')$$
\end{prop}
\begin{proof}
  We need the use of ribbon graphs with coupons.  We start with the
  computable link presentations $(L,K,\sigma)$ and color the
  components of $L_i$ with Kirby colors.  Up to isotopy, we can assume
  that the disc $D$ is in an horizontal plane and intersects $n$
  vertical strands colored by some modules $V_1,\ldots,V_n$.  These
  $n$ strands represent the identity of $W=V_1'\otimes\cdots\otimes
  V_n'$ where $V'_i$ is $V_i$ or $V_i^*$ according to the orientation
  of the $i^{th}$ strand.  By hypothesis, $W\in\cat_{\wb\alpha}$.
  Then the following modifications does not change the value of the
  invariant by $F'$:
  \begin{enumerate}
  \item Replace a neighborhood of $D$ by a unique vertical strand
    colored by $W$ and connected to the $n$-strands above and ahead
    the disc $D$ by two coupons colored by the morphism $\Id_W$.
  \item Make a KI$^+$-move followed by a KI$^-$-move on the
    $W$-colored strand.  Here we directly color the new components
    $K_\pm$ with a Kirby color divided by $|\Delta_+^\Spin|$. The two
    changes of the framing of the $W$-colored strand compensate.
  \item Remove the $W$-colored strand and the two coupons and replace
    it back with the $n$ original strands.
  \end{enumerate}
  This sequence of modifications lead to $F'(L,K\cup K_+\cup
  K_-,\sigma')$ as describe by a birth move.
\end{proof}
\section{The invariant}\label{S:TheInv}
\subsection{Main theorem}
In this subsection we define the invariant and give some of its properties.
The proofs are postponed until Subsection \ref{SS:ProofOfInv}.

Recall that if $F_KM$ is the framed oriented bundle of $M\setminus K$, a
\spinst\ on $M\setminus K$ is a cohomology class $\sigma$ in
$H^1(F_KM;\C/2\Z)$.  We will say that $\sigma$ is \emph{integral} if $\sigma$
belongs to $\Hom(H_1(F_KM;\Z);\Z/2\Z)$, via the universal coefficients
theorem: $H^1(F_KM;\C/2\Z)\cong \Hom(H_1(F_KM;\Z);\C/2\Z)$.  This means that
$\sigma$ is a natural spin structure on $M\setminus K$ (associated to the
group $\Spin(n)=\Spin(n,\Z/2\Z)$).

\begin{defi}
  A compatible triple $(M,K,\sigma)$ is admissible if $K$ has a
  component colored by $\alpha\in\Cp$ or if $\sigma$ is a non-integral
  \spinst\ on $M\setminus K$.
\end{defi}

\begin{theo}\label{T:main}
 Let $(M,K,\sigma)$ be an admissible triple.  \\
 (1)   If $\sigma$ is not
  integral then there exists a computable surgery presentation
  $(L,K,\sigma)$ of $(M,K,\sigma)$. If $\sigma$ is integral, there
  exists a computable surgery presentation $(L,K',\sigma')$ of an Hopf
  stabilization of $(M,K,\sigma)$.\\
(2) In both cases of part $(1)$ we have, 
  $$\Nr(M,K,\sigma)={\Delta(L)}{F'(L,K,\sigma)}
  \quad\big(\text{respectively }
  \Nr(M,K,\sigma)={\Delta(L)}F'(L,K',\sigma')\big)$$ is an invariant
  of the diffeomorphism class of $(M,K,\sigma)$, where 
  $\Delta(L)=(\Delta^\Spin_+)^{-b_+}(\Delta^\Spin_-)^{-b_-}$ with
  $(b_+,b_-)$ being the signature of the linking matrix of $L$.
\end{theo}
The proof of this theorem will be given in the next section.
We now state some properties of the invariant $\Nr$.

If two admissible triple $(M,K\cup K_\alpha,\sigma)$ and $(M',K'\cup
K_\alpha',\sigma')$ have a distinguished component colored by the same
$\alpha\in\Cp$, then one can consider their banded connected sum
$(M,K\cup K_\alpha,\sigma)\#_\alpha(M',K'\cup
K_\alpha',\sigma')=(M\#M',K\cup (K_\alpha\#K_\alpha'),\sigma \cup
\sigma')$ obtained by removing from $M$ a small 3-ball intersecting
$K_\alpha$, removing from $M'$ a small 3-ball intersecting $K'_\alpha$
and gluing these two manifolds along their diffeomorphic sphere. Then
we have
\begin{prop}[Banded connected sum] Under the above hypothesis we have 
  $$\Nr((M,K\cup K_\alpha,\sigma)\#_\alpha(M',K'\cup
  K_\alpha',\sigma')=\qd(\alpha)^{-1}\Nr(M,K\cup
  K_\alpha,\sigma)\Nr(M',K'\cup K_\alpha',\sigma').$$
\end{prop}
For the ordinary connected sum of compatible triples, remark that the
result of the connected sum of an admissible triple $(M,K,\sigma)$
with any compatible $(M',K',\sigma')$ is admissible.  This allows to
compute a secondary invariant $\Nr^0$ for non admissible triples:
\begin{theo}[$\Nr^0$ and connected sum]
  There exists a unique $\C$-valued invariant $\Nr^0$ defined on
  any compatible triple (not necessary admissible) which is zero on
  admissible triples and such that for all admissible triple
  $(M,K,\sigma)$ and any not necessary admissible triple
  $(M',K',\sigma')$,
  $$\Nr((M,K,\sigma)\#(M',K',\sigma'))=\Nr(M,K,\sigma)\Nr^0(M',K',\sigma').$$
\end{theo}
\subsection{Proof of the invariance}\label{SS:ProofOfInv}
The moves of Section \ref{S:moves} can be considered as the edges a graph
$\Gamma$ whose vertices are link presentations.  Let $\Gamma_0$ be the
subgraph of $\Gamma$ consisting of computable presentations.  Then the propositions
of Section \ref{S:moves} imply that
\begin{prop}
  The function $(L,K,\sigma)\mapsto{\Delta(L)}{F'(L,K,\sigma)}$ is
  constant on the connected components of $\Gamma_0$.
\end{prop}
\begin{proof}
  The factors ${\Delta(L)}$ and $F'(L,K,\sigma)$ are invariant by all
  the moves except for KI$^\pm$ moves for which they exactly
  compensates.  Indeed, a positive KI$^\pm$ move adds exactly $1$ to
  the signature integer $b_\pm$.
\end{proof}

\begin{proof}[Proof Theorem \ref{T:main} Part (1)]
 Here we prove the existence of computable presentations.   Let $(M,K,\sigma)$ be admissible as in Theorem \ref{T:main} Part (1).  If
  $\sigma$ is integral, then $K$ has a component colored by an element
  of $r\Z\subset\Cp$.  Then we replace $(M,K,\sigma)$ by an Hopf
  stabilization on such an edge so that the added meridian has a color
  in $\Cp\setminus\Z$.  

  Now we are reduced to the case where $\sigma$ has a non integral
  value.  Consider any surgery presentation $(L,K,\sigma)$ of
  $(M,K,\sigma)$ and consider a component $J$ of $L\cup K$ with
  $c_\sigma(J)\notin\Z/2\Z$. 
   If $c_\sigma(L_i)\in\Z/2\Z$ for some component $L_i$ of $L$ 
  then we do a KII move by sliding $J$
  on $L_i$.  In the resulting link $L'$, we have
  $c_\sigma(L_i')\notin\Z/2\Z$
   where $L_i'$ corresponds to $L_i$
   and the other values of $c_\sigma$ are
  unchanged.  Repeating this for all components of $L$ if necessary
  leads to a computable presentation of $(M,K,\sigma)$.
  \renewcommand{\qedsymbol}{\fbox{\ref{T:main}(1)}}
\end{proof}
\renewcommand{\qedsymbol}{\fbox{\thetheo}}

\begin{rem}
  It is clear that if $(M_1,K_1,\sigma_1)$ and $(M_2,K_2,\sigma_2)$
  are two Hopf stabilizations of $(M,K,\sigma)$, then there exists
  $(M_{12},K_{12},\sigma_{12})$ which is both a Hopf stabilization of
  $(M_1,K_1,\sigma_1)$ and $(M_2,K_2,\sigma_2)$.  
\end{rem}
With this remark, Theorem \ref{T:main} Part (2) follows from the following
proposition whose proof is essentially a reformulation of \cite{CGP1}:
\begin{prop}\label{P:computableconnected}
  Any two computable link presentations of an admissible triple
  $(M,K,\sigma)$ correspond to vertices of the same connected
  component of $\Gamma_0$.
\end{prop}
\begin{proof}
  The key point is Kirby's theorem and its refinement:
  \begin{theo}[see Theorem 5.2 of \cite{CGP1}]\label{T:CGP5.2}
    Let $L\cup K$ and $L'\cup K'$ be two link in $S^3$ and
    $f:(S^3(L),K)\to(S^3(L'),K')$ be an orientation preserving
    diffeomorphism.  Then there is a sequence of orientation moves,
    KI$^\pm$ moves and KII moves transforming $(L,K)$ to $(L',K')$
    whose associated canonical diffeomorphism is $f$.
  \end{theo}
  We can use this theorem in the case a diffeomorphism
  $g:(M,K,\sigma)\to(M',K',\sigma')$ respects the \spinst.  Then if
  $(L,K,\sigma)$ is a presentation of $(M,K,\sigma)$ and
  $(L',K',\sigma')$ is a presentation of $(M',K',\sigma')$, $g$
  induces a diffeomorphism $f$ as in Theorem~\ref{T:CGP5.2} that will send the
  \spinst\ of $S^3(L)\setminus K$ to the \spinst\ of $S^3(L')\setminus
  K'$.  As a consequence, there exists a sequence of orientation
  moves, KI$^\pm$ moves and KII moves connecting the link
  presentations $(L,K,\sigma)$ and $(L',K',\sigma')$ (with their
  \spinst).  As a corollary, we have
  \begin{theo}\label{Kirby-spin}
    $(L,K,\sigma)$ and $(L',K',\sigma')$ are link presentations of
    positively diffeomorphic compatible triples if and only if they
    are connected by a finite sequence of orientation moves, KI$^\pm$
    moves and KII moves.
  \end{theo}
  Let $\wb L=(L,K,\sigma)$ and $\wb L'=(L',K',\sigma')$ be two computable link presentations  of the same admissible
  triple.  By the Theorem \ref{Kirby-spin}, there is a finite sequence of
  orientation moves, KI$^\pm$ moves and KII moves connecting $\wb L$ and $\wb L'$.

  To prove Theorem \ref{T:main} Part (2), we first reduce to the case when $K$ has a component colored by an
  element of $\Cp$: If it is not the case, then there is a component
  $L_i$ of $L$ with $c_\sigma(L_i)\notin\Z/2\Z$.  We consider a small
  disc in $S^3$ that intersects $L_i$ once.  We do a birth move on this
  disc that creates two new components.  Then we can perform the
  analog of the finite sequence of orientation moves, KI$^\pm$ moves
  and KII moves on this link presentation and this leads to a link
  presentation related to $\wb L'$ by a death move (the two created
  components stay parallel and unknotted in $S^3$ during all this
  process).  Hence up to replacing $\wb L$ and $\wb L'$ by these
  computable presentations obtained from them by a birth move, we can
  assume that $K$ has a component (say $K_1$) colored by a element of
  $\Cp$.

  Second, we do two Hopf stabilizations on $K_1$, creating two
  meridians $J_\alpha$ and $J_\beta$ colored with some generic color
  $\alpha,\beta\in\Cp$ (by generic, we mean that $\alpha$ and $\beta$
  are rationally independent with all the values of $\sigma$ or
  $c_\sigma$).  Now sliding $J_\alpha$ and $J_\beta$ on the surgery
  components allows us to change arbitrarily the $\C/2\Z$-coloring of
  surgery components by adding any element of $\Z\alpha+\Z\beta$
  through computable presentation (this is the reason of the use of
  two meridians with independent colors).

  Now we start performing the analog of the finite sequence of
  orientation moves, KI$^\pm$ moves and KII moves ignoring $J_\alpha$
  and $J_\beta$ but during the sequence,
  \begin{itemize}
  \item if a KII move which is a sliding on $L_i$ generates an integral
    value of $c_\sigma(L_i')$, we first change the color of $L_i$ by
    an element of $\Z\alpha+\Z\beta$ so that the KII move does not
    cause the appearance of a color in $\Z/2\Z$ ;
  \item if a KI$^\pm$ move is performed around an edge of $K$ of integral
    degree, we do it instead on $J_\alpha$ and then slide the edge of
    $K$ on the created component.  
  \end{itemize}
  Doing this we follow a path in $\Gamma_0$ leading to $\wb L'$ union
  two knots $J_\alpha$ and $J_\beta$ that may be linked with other
  components of $L'\cup K'$.  But in $S^3(L')\setminus K'$, $J_\alpha$
  and $J_\beta$ are meridians of a component (say $K'_1$) of $K'$.
  Thus there exists an isotopy in $S^3(L')\setminus K'$ moving
  $J_\alpha$ and $J_\beta$ to a small neighborhood of $K'_1$.  To this
  isotopy corresponds a sequence of KII moves where only $J_\alpha$ and
  $J_\beta$ are sliding on the surgery components.  This sequence leads
  to a double Hopf stabilization of $\wb L'$.  Furthermore, as a color
  of a surgery component $L_i'$ during these moves belongs to
  $c_{\sigma'}+\alpha\Z+\beta\Z$ modulo 2 it is never an integer and
  all the sequence is in $\Gamma_0$.
  \renewcommand{\qedsymbol}{\fbox{\ref{P:computableconnected}}}
\end{proof}
\renewcommand{\qedsymbol}{\fbox{\thetheo}}

\end{document}